\documentclass[12pt]{amsart}
\usepackage{cite,graphics,mathptmx}
\newcommand{\set}[1]{\left\{#1\right\}}

\numberwithin{equation}{section}

\newtheorem{theorem}{Theorem}[section]

\theoremstyle{remark}
 \newtheorem{remark}[theorem]{Remark}

\newtheorem{definition}[theorem]{Definition}

\begin{document}
 
\title{Estimates for the Initial Coefficients of Bi-univalent Functions}
\author[S. Sivaprasad Kumar]{S. Sivaprasad Kumar}

\address{Department of Applied Mathematics,
Delhi Technological University, Delhi---110042, India}
\email{spkumar@dce.ac.in}

\author[V. Kumar]{Virendra Kumar}
\address{Department of Applied Mathematics,
Delhi Technological University, Delhi---110042, India}
\email{vktmaths@yahoo.in}

\author{V. Ravichandran}
\address{Department of Mathematics, University of Delhi, Delhi---110007, India}
\email{vravi@maths.du.ac.in}

\begin{abstract}A bi-univalent function is a univalent function defined on the  unit disk with its inverse also univalent on the unit disk.  In the present investigation, estimates for the initial coefficients are obtained  for   bi-univalent functions belonging to certain classes defined by subordination and  relevant connections with earlier results are pointed out.
\end{abstract}



\maketitle

\section{Introduction}
Let $\mathcal{A}$ be the class of analytic functions defined  on the open unit disk $\mathbb{D}:=\{z\in\mathbb{C}:|z|<1\}$ and normalized by the conditions $f(0)=0$ and $f'(0)=1$. A function $f\in \mathcal{A}$  has Taylor's series expansion of the form \begin{equation}\label{eq1}
f(z)=z+\sum^{\infty}_{n=2}a_nz^n.\end{equation}
The class of all  univalent functions in the open unit disk $\mathbb{D}$ of the form (\ref{eq1})
 is denoted by $\mathcal{S}$. Determination of the bounds for  the coefficients $a_n$  is an important problem in
 geometric function theory as they give information about the geometric properties of these functions.
  For example, the bound for the second coefficient $a_2$ of  functions in $\mathcal{S}$ gives the growth
  and distortion bounds as well as covering theorems. Some coefficient related problems were investigated
  recently in \cite{filomat,ali11,chen,cho10,liu11,samni}. 

Since univalent functions are one-to-one, they are invertible but their inverse functions need not be defined on the entire unit disk $\mathbb{D}$. In fact, the famous Koebe one-quarter theorem ensures that the image of the unit disk $\mathbb{D}$ under every function $f\in\mathcal{S}$ contains a disk of radius $1/4$. Thus, inverse of every function $f\in\mathcal{S}$ is defined on a disk, which contains the disk $|z|<1/4$. It can also be easily verified that
\begin{equation}\label{ei}F(w):=f^{-1}(w)=w-a_2w^2+(2a_2^2-a_3)w^3-(5a_2^2-5a_2a_3+a_4)w^4+
\cdots\end{equation}
in some disk of radius at least 1/4.  A function $f\in\mathcal{A}$ is called \emph{bi-univalent}  in $\mathbb{D}$ if both $f$ and $f^{-1}$ are univalent in $\mathbb{D}$. In 1967, Lewin \cite{lewin} introduced the class $\sigma$ of bi-univalent analytic functions and showed that the second coefficient of every $f\in\sigma$  satisfy the inequality $|a_2|\leq1.51$. Let $\sigma_1$ be the class of all functions $f=\phi\circ\psi^{-1}$ where $\phi,\psi$ map $\mathbb{D}$ onto a domain containing $ \mathbb{D}$ and $\phi'(0)=\psi'(0)$. In 1969, Suffridge \cite{suff69} gave a function in $\sigma_1\subset\sigma$, satisfying $a_2=4/3$ and conjectured that $|a_2|\leq 4/3$ for all functions in $\sigma$. In 1969, Netanyahu \cite{neta} proved this conjecture for the subclass $\sigma_1$. Later in 1981, Styer and Wright \cite{sty81} disproved the conjecture of Suffridge \cite{suff69} by showing $a_2>4/3$ for some function in $\sigma$. Also see \cite{bhs} for an example to show $\sigma\neq \sigma_1$. For results on bi-univalent polynomial, see \cite{smith1,ked88}. In 1967, Brannan \cite{branan} conjectured that $|a_2|\leq \sqrt{2}$ for $f\in \sigma$. In 1985, Kedzierawski {\cite [Theorem 2]{ked89}} proved this conjecture for a special case when the function  $f $  and $f^{-1}$  are starlike functions. In $1985$, Tan \cite{tan} obtained the bound for $a_2$ namely $|a_2|\leq 1.485$ which is the best known estimate for functions in the class $\sigma$. For some open problems and survey, see \cite{good,smith2}. In $1985$, Kedzierawski \cite{ked89} proved the following:
 $$|a_2|\leq\left\{
   \begin{array}{ll}
     1.5894, & \hbox{$f\in \mathcal{S},\; f^{-1}\in \mathcal{S}$;} \\
     \sqrt2, & \hbox{$f\in \mathcal{S}^*,\; f^{-1}\in \mathcal{S}^*$;} \\
     1.507, & \hbox{$f\in \mathcal{S}^*,\; f^{-1}\in \mathcal{S}$;} \\
     1.224, & \hbox{$f\in \mathcal{K},\; f^{-1}\in \mathcal{S}$,}
   \end{array}
 \right.$$
where $\mathcal{S}^*$ and $\mathcal{K}$ denote the well-known classes of starlike and convex functions in $\mathcal{S}.$

 Let us recall now various definitions required in sequel. An analytic  function $f$ is  \emph{subordinate} to another analytic function $g$, written $f\prec g$, if there is an analytic function $w$ with $|w(z)|\leq |z|$ such that $f =g\circ w$. If $g$ is univalent, then $f\prec g$ if and only if $f(0)=g(0)$ and $f(\mathbb{D})\subseteq g(\mathbb{D})$. Let $\varphi$ be an analytic univalent function in $\mathbb{D}$ with positive real part  and $\varphi(\mathbb{D})$ be symmetric with respect to the real axis, starlike with respect to $\varphi(0)=1$  and $\varphi'(0)>0.$ Ma and Minda \cite{minda} gave a unified presentation of various subclasses of starlike and convex functions by introducing the classes $\mathcal{S}^*(\varphi)$ and $\mathcal{K}(\varphi)$ of functions $f\in \mathcal{S}$ satisfying $zf'(z)/f(z)\prec\varphi(z)$ and $1+zf''(z)/f'(z)\prec\varphi(z)$ respectively, which includes several well-known classes as special case. For example, when $\varphi(z)=(1+A z)/(1+Bz)\; (-1\leq B<A\leq 1),$ the class $\mathcal{S}^*(\varphi)$ reduces to the class $\mathcal{S}^*[A,B]$ introduced by Janowski \cite{jano}.  For $0\leq\beta<1$, the classes $\mathcal{S}^*(\beta):=\mathcal{S}^*((1+(1-2\beta)z)/(1-z))$
 and $\mathcal{K}(\beta):=\mathcal{K}((1+(1-2\beta)z)/(1-z))$ are starlike and convex functions of order $\beta.$
Further let $\mathcal{S}^*:=\mathcal{S}^*(0)$ and $\mathcal{K}:=\mathcal{K}(0)$ are the classes of starlike and convex functions respectively. The class of strongly starlike functions $\mathcal{S}^*_\alpha:=\mathcal{S}^*(((1+z)/(1-z))^\alpha)$ of order $\alpha,\; 0<\alpha\leq1.$ Denote by $\mathcal{R}(\varphi)$ the class of all functions satisfying $f'(z)\prec\varphi(z)$ and let $\mathcal{R}(\beta):=\mathcal{R}((1+(1-2\beta)z)/(1-z))$ and $\mathcal{R}:=\mathcal{R}(0)$.

For $ 0\leq\beta<1$, a function $f\in \sigma$ is in the class $\mathcal{S}^*_{\sigma}(\beta)$ of \emph{bi-starlike  function of order $\beta$}, or $\mathcal{K}_{\sigma}(\beta)$ of\emph{ bi-convex function of order $\beta$} if both $f$ and $f^{-1}$ are respectively starlike or convex functions of order $\beta$. For $0<\alpha\leq 1$,  the function $f\in\sigma$ is   \emph{strongly bi-starlike function of order $\alpha$} if both the functions $f$ and $f^{-1}$ are strongly starlike functions of order $\alpha$. The class of all such functions is denoted by $\mathcal{S}^*_{\sigma,\alpha}$. These classes were introduced by  Brannan and Taha~\cite{branan1} in 1985 (see also \cite{branan0}). They obtained estimates on the initial coefficients $a_2$ and $a_3$ for functions in these classes. Recently, Ali \emph{et al.}\ \cite{ravi} extended the results of   Brannan and Taha \cite{branan1}  by generalizing their classes using subordination. For some related results, see \cite{srivastava,frasin,xu}. For the various applications of subordination one can refer to \cite{filomat,ali11,cho10,liu11,samni} and the references cited therein.

Motivated by  Ali \emph{et al.}\ \cite{ravi} in this paper estimates for the initial coefficient $a_2$ of bi-univalent functions belonging to the  class $\mathcal{R}_\sigma(\lambda, \varphi)$ as well as estimates on $a_2$ and $a_3$ for functions in classes $\mathcal{S^*}_\sigma(\varphi)$ and $K_{\sigma}(\varphi)$, defined later, are obtained. Further work of Kedzierawski \cite{ked89} actuates us to derive the estimates on  initial coefficients $a_2$ and $a_3$ when $f$ is in the some subclass of univalent functions and $f^{-1}$ belongs to some other subclass of univalent functions. Our results generalize several well-known results in \cite{ravi, frasin,ked89,srivastava}, which are pointed out here.

\section{Coefficient estimates}
Throughout this paper, we assume that  $\varphi$ is an analytic function in $\mathbb{D}$ of the form \begin{equation}\label{eb}\varphi(z)=1+B_1z+B_2z^2+B_3z^3+\cdots \;\;\text{with}\;\;B_1>0,\;{\text {and}}\; B_2\;{\text{ is any real number}} .\end{equation}

\begin{definition}Let $\lambda\geq0$.  A function $f\in \sigma$ given by (\ref{eq1}) is  in the class $\mathcal{R}_\sigma(\lambda, \varphi)$,  if it satisfies
\[(1-\lambda)\frac{f(z)}{z}+\lambda f'(z)\prec\varphi(z) \quad \text{ and}\quad  (1-\lambda)\frac{F(w)}{w}+\lambda F'(w)\prec\varphi(w).\] \end{definition}

The class $\mathcal{R}_\sigma(\lambda, \varphi)$ includes many earlier classes, which are mentioned below:
\begin{enumerate}
\item $\mathcal{R}_\sigma(\lambda, (1+(1-2\beta)z)/(1-z))=\mathcal{R}_{\sigma}(\lambda,\beta)\;$  $(\lambda\geq 1;\;0\leq\beta<1)$    \cite[Definition 3.1]{frasin}

\item $\mathcal{R}_\sigma(\lambda, ((1+z)/(1-z))^\alpha)=\mathcal{R}_{\sigma,\alpha}(\lambda)\;\; (\lambda\geq 1;\;0<\alpha\leq1)$    \cite[Definition 2.1]{frasin}

\item $\mathcal{R}_\sigma(1,\varphi)
      =\mathcal{R}_{\sigma}(\varphi)$  \cite[p. 345]{ravi}.
\item $\mathcal{R}_\sigma(1,(1+(1-2\beta)z)/(1-z))
      =\mathcal{R}_\sigma(\beta)\;\; (0\leq\beta<1)$  \cite[Definition 2]{srivastava}
\item  $\mathcal{R}_\sigma(1,((1+z)/(1-z))^\alpha)
      =\mathcal{R}_{\sigma,\alpha}\;\; (0<\alpha\leq1)$  \cite[Definition 1]{srivastava}

\end{enumerate}
 Our first result provides estimate for the coefficient $a_2$ of functions $f\in\mathcal{R}_\sigma(\lambda, \varphi)$.

\begin{theorem}\label{th1} If $f\in\mathcal{R}_\sigma(\lambda, \varphi)$, then
\begin{align}\label{e1}|a_2|&\leq \sqrt\frac{B_1+|B_1-B_2|}{1+2\lambda}
\end{align}
\end{theorem}

\begin{proof}Since $f\in\mathcal{R}_\sigma(\lambda, \varphi)$,   there exist two analytic functions $r, s:\mathbb{D}\rightarrow\mathbb{D}$,  with $r(0)=0=s(0)$,  such that
\begin{equation}\label{p1}(1-\lambda)\frac{f(z)}{z}+\lambda f'(z)=\varphi(r(z))\;\text {and}\; (1-\lambda)\frac{F(w)}{w}+\lambda F'(w)=\varphi(s(z)).\end{equation}
Define the functions $p$ and $q$
 by
 {\small\begin{equation}\label{pq}
    p(z)=\frac{1+r(z)}{1-r(z)}= 1+p_1z+p_2z^2+p_3z^3+\cdots \;\;\text {and}\;\; q(z)=\frac{1+s(z)}{1-s(z)}= 1+q_1z+q_2z^2+q_3z^3+\cdots,
 \end{equation}}
 or equivalently,
 \begin{equation}\label{p2}
  r(z)=\frac{p(z)-1}{p(z)+1}=\frac{1}{2}\left(p_1z+\Big(p_2-\frac{p_1^2}{2}\Big)z^2+
  \Big(p_3+\frac{p_1}{2}\big(\frac{p_1^2}{2}-p_2\big)-\frac{p_1p_2}{2}\Big)z^3+\cdots\right)
 \end{equation}
 and
\begin{equation}\label{p3}
  s(z)=\frac{q(z)-1}{q(z)+1}=\frac{1}{2}\left(q_1z+\Big(q_2-\frac{q_1^2}{2}\Big)z^2+ \Big(q_3+\frac{q_1}{2}\big(\frac{q_1^2}{2}-q_2\big)-\frac{q_1q_2}{2}\Big)z^3+\cdots\right).
 \end{equation}
It is clear that $p$ and $q$ are analytic in $\mathbb{D}$ and $p(0)=1=q(0)$. Also $p$ and $q$  have positive real part in $\mathbb{D}$, and hence $|p_i|\leq2$ and  $|q_i|\leq2$. In the view of (\ref{p1}), (\ref{p2}) and (\ref{p3}), clearly
 \begin{equation}\label{p4}(1-\lambda)\frac{f(z)}{z}+\lambda f'(z)=\varphi\left(\frac{p(z)-1}{p(z)+1}\right)\;\;\text {and}\;\; (1-\lambda)\frac{F(w)}{w}+\lambda F'(w)=\varphi\left(\frac{q(w)-1}{q(w)+1}\right).\end{equation}
On expanding (\ref{eb}) using (\ref{p2}) and (\ref{p3}), it is evident that
\begin{multline}\label{p5}
\varphi\left(\frac{p(z)-1}{p(z)+1}\right)= 1+\frac{1}{2}B_1p_1z+\left(\frac{1}{2}B_1\big(p_2-\frac{1}{2}p_1^2\big)
   +\frac{1}{4}B_2p_1^2\right)z^2+\cdots.
\end{multline}
and
\begin{multline}\label{p6}
\varphi\left(\frac{q(w)-1}{q(w)+1}\right)= 1+\frac{1}{2}B_1q_1w+\left(\frac{1}{2}B_1\big(q_2-\frac{1}{2}q_1^2\big)
   +\frac{1}{4}B_2q_1^2\right)w^2+\cdots.
\end{multline}
Since $f\in\sigma$ has the Maclaurin series given by (\ref{eq1}), a computation shows that its inverse $F=f^{-1}$ has the expansion given by (\ref{ei}). It follows from (\ref{p4}), (\ref{p5}) and (\ref{p6}) that
\[(1+\lambda) a_2=\frac{1}{2}B_1p_1,\]
\begin{equation}\label{eq12}
 (1+2\lambda) a_3 =\frac{1}{2}B_1\Big(p_2-\frac{1}{2}p_1^2\Big)+\frac{1}{4}B_2p_1^2,
\end{equation}
\[ -(1+\lambda) a_2=\frac{1}{2}B_1q_1,\]
\begin{equation}\label{eq16}
 (1+2\lambda) (2a_2^2-a_3) =\frac{1}{2}B_1\left(q_2-\frac{1}{2}q_1^2\right)+\frac{1}{4}B_2q_1^2.
\end{equation}
Now  (\ref{eq12}) and (\ref{eq16}) yield
 \begin{equation}\label{eq17a}
 8(1+2\lambda)a_2^2=2(p_2+q_2)B_1+(B_2-B_1)(p_1^2+q_1^2).
\end{equation}
Finally an application of the known results, $|p_i|\leq2$ and  $|q_i|\leq2$
 in (\ref{eq17a}) yields the desired estimate of $a_2$ given by (\ref{e1}).
\end{proof}
\begin{remark} Let $\varphi(z)=(1+(1-2\beta)z)/(1-z),\; 0\leq\beta<1.$ So $B_1=B_2=2(1-\beta).$
 When $\lambda=1$, Theorem \ref{th1} gives the estimate $|a_2|\leq \sqrt {2(1-\beta)/3}$ for functions in the class $\mathcal{R}_\sigma(\beta)$ which coincides with the result \cite[Corollary 2]{xu} of Xu \emph{et al}.
In particular if $\beta=0$, then above estimate becomes $|a_2|\leq \sqrt{2/3}\approx 0.816$ for functions $f\in \mathcal{R}_\sigma(0).$ Since the estimate on $|a_2|$ for $f\in  \mathcal{R}_\sigma(0) $ is improved over the conjectured estimate   $|a_2|\leq\sqrt{2}\approx 1.414$ for $f\in \sigma$, the functions in $ \mathcal{R}_\sigma(0) $ are not the candidate for the sharpness of the estimate in the class $\sigma$.
\end{remark}
\begin{definition}  A function $f\in \sigma$  is  in the class $\mathcal{S^*}_\sigma(\varphi)$,  if it satisfies
 $$\frac{zf'(z)}{f(z)}\prec\varphi(z)\quad \text{and}\quad \frac{wF'(w)}{F(w)}\prec\varphi(w).$$
 \end{definition}
 Note that for a suitable choice of $\varphi$,  the class $\mathcal{S^*}_\sigma(\varphi)$, reduces to the following well-known classes:
 \begin{enumerate}
   \item $\mathcal{S^*}_\sigma((1+(1-2\beta)z)/(1-z))=\mathcal{S^*}_\sigma(\beta) \quad(0\leq\beta<1).$
   \item $\mathcal{S^*}_\sigma(\left((1+z)/(1-z)\right)^\alpha)
    =\mathcal{S}^*_{\sigma,\alpha} \quad(0<\alpha\leq 1).$
 \end{enumerate}
\begin{theorem} \label{th2}If $f\in\mathcal{S^*}_\sigma(\varphi)$, then
$$|a_2|\leq\min\set{\sqrt{B_1+|B_2-B_1|}, \sqrt{ \frac{B_1^2+B_1+|B_2-B_1|}{2}}, \frac{B_1\sqrt{B_1}}{\sqrt{B_1^2+|B_1-B_2|}}}$$ and
$$|a_3|\leq\min\set{ B_1+|B_2-B_1|, \frac{B_1^2+B_1+|B_2-B_1|}{2}, R},$$ where
$$R:= \frac{1}{4}\left(B_1+3B_1\max\set{1;\left|\frac{B_1-4B_2}{3B_1}\right|}\right).$$
\end{theorem}

\begin{proof} Since $f\in \mathcal{S}^*_\sigma(\varphi)$,  there are analytic functions $r, s:\mathbb{D}\rightarrow\mathbb{D}$,  with $r(0)=0=s(0)$,  such that
\begin{equation}\label{p1.1}\frac{zf'(z)}{f(z)}=\varphi(r(z))\;\text {and}\; \frac{wF'(w)}{F(w)} =\varphi(s(z)).\end{equation}
Let $p$ and $q$ be defined as in (\ref{pq}),
then it is clear from (\ref{p1.1}), (\ref{p2}) and (\ref{p3}) that
\begin{equation}\label{p4.1}\frac{zf'(z)}{f(z)}=\varphi\left(\frac{p(z)-1}{p(z)+1}\right)\;\;\text {and}\;\; \frac{wF'(w)}{F(w)}=\varphi\left(\frac{q(z)-1}{q(z)+1}\right).\end{equation}
It follows from (\ref{p4.1}), (\ref{p5}) and (\ref{p6}) that
\begin{equation}\label{eq11.1}
 a_2=\frac{1}{2}B_1p_1,
\end{equation}
\begin{equation}\label{eq12.1}
 2a_3 =\frac{B_1p_1}{2}a_2+\frac{1}{2}B_1\left(p_2-\frac{1}{2}p_1^2\right)+\frac{1}{4}B_2p_1^2,
\end{equation}
\begin{equation}\label{eq15.1}
 -a_2=\frac{1}{2}B_1q_1
\end{equation}
and
\begin{equation}\label{eq16.1}
 4a_2^2-2a_3=-\frac{B_1q_1}{2}a_2+
 \frac{1}{2}B_1\left(q_2-\frac{1}{2}q_1^2\right)+\frac{1}{4}B_2q_1^2.
\end{equation}
The equations (\ref{eq11.1}) and (\ref{eq15.1}) yield
\begin{equation}\label{eq17.1}
 p_1=-q_1,
\end{equation}
\begin{equation}\label{eq18.10}
 8a_2^2=(p_1^2+q_1^2)B_1^2
\end{equation}
and
\begin{equation}\label{eq18.1}
 2a_2=\frac{B_1(p_1-q_1)}{2}.
\end{equation}
From (\ref{eq12.1}), (\ref{eq16.1}) and (\ref{eq18.1}), it follows that
\begin{equation}\label{eq19}
 8a_2^2=2B_1(p_2+q_2)+(B_2-B_1)(p_1^2+q_1^2).
\end{equation}
Further a computation using (\ref{eq12.1}), (\ref{eq16.1}), (\ref{eq11.1}) and (\ref{eq17.1}) gives
\begin{equation}\label{eq19.1}
 16a_2^2=2B_1^2q_1^2+2B_1(p_2+q_2)+(B_2-B_1)(p_1^2+q_1^2).
\end{equation}
Similarly a computation using (\ref{eq12.1}), (\ref{eq16.1}), (\ref{eq18.1}) and (\ref{eq18.10}) yields
\begin{equation}\label{eq19.10}
 4(B_1^2-B_2+B_1)a_2^2=B_1^3(p_2+q_2).
\end{equation}
Now (\ref{eq19}), (\ref{eq19.1}) and (\ref{eq19.10}) yield the desired estimate on $a_2$ as asserted in the theorem. To find estimate for $a_3$ subtract (\ref{eq12.1}) from (\ref{eq16.1}), to get
\begin{equation}\label{eq19.2}
 -4a_3=-4a_2^2+\frac{B_1(q_2-p_2)}{2}.
\end{equation}
 Now a computation using (\ref{eq19.1}) and (\ref{eq19.2}) leads to
\begin{equation}\label{eq19.31}
 16a_3=2B_1^2q_1^2+4B_2p_2+(B_1-B_2)(p_1^2+q_1^2).
\end{equation}
From (\ref{eq11.1}), (\ref{eq12.1}), (\ref{eq15.1}) and (\ref{eq16.1}), it follows that

\begin{eqnarray}\label{eq19.33}
  4a_3 &=& \frac{B_1}{2}(3p_2+q_2)+(B_2-B_1)p_1^2 \\ \label{eq19.3}
   &=& \frac{B_1q_2}{2}+\frac{3B_1}{2}\left(p_2-\frac{2(B_1-B_2)}{3B_1}p_1^2\right).
\end{eqnarray} On applying the result of Keogh and Merkes \cite{koef69}(see also \cite{ravi05}), that is for any complex number $v$, $|p_2-vp_1^2|\leq 2 \max\{1;|2v-1|\}$, along with $|q_2|\leq2$ in (\ref{eq19.3}), we obtain
\begin{equation}\label{eq19.30}
 4|a_3|\leq B_1+3B_1\max\set{1;\left|\frac{B_1-4B_2}{3B_1}\right|}.
\end{equation}
 Now the desired estimate on $a_3$ follows from (\ref{eq19.31}), (\ref{eq19.33}) and (\ref{eq19.30}) at once.\end{proof}
\begin{remark}  If $f\in\mathcal{S^*}_\sigma(\beta)\;\;(0\leq\beta<1)$, then from Theorem~\ref{th2} it is evident that

{\small\begin{equation}\label{1}
|a_2|\leq\min\set{\sqrt{2(1-\beta)}, \sqrt{(1-\beta)(3-2\beta)}}=\left\{
  \begin{array}{ll}
    \sqrt{2(1-\beta)}, & \hbox{$0\leq\beta\leq1/2$;} \\
    \sqrt{(1-\beta)(3-2\beta)}, & \hbox{$1/2\leq\beta<1$.}
  \end{array}
\right.
\end{equation}}
Recall Brannan and Taha's~\cite[Theorem 3.1]{branan0} coefficient estimate, $|a_2|\leq \sqrt{2(1-\beta)}$ for functions $f\in\mathcal{S^*}_\sigma(\beta),$ who claimed that their estimate is better than the estimate $|a_2|\leq 2(1-\beta)$, given by Robertson~\cite{rob}. But their claim is true only when $0\leq\beta\leq 1/2.$ Also it may noted that our estimate for $a_2$ given in (\ref{1}) improves the estimate given by Brannan and Taha~\cite[Theorem 3.1]{branan0}.

Further if we take $\varphi(z)=((1+z)/(1-z))^\alpha,\;0<\alpha\leq1$ in Theorem~\ref{th2}, we have $B_1=2\alpha$ and $B_2=2\alpha^2$. Then we obtain the estimate on $a_2$ for functions $f\in\mathcal{S^*}_{\sigma,\alpha}$ as:
$$|a_2|\leq\min\set{\sqrt{4\alpha-2\alpha^2}, \sqrt{\alpha^2+2\alpha}, \frac{2\alpha}{\sqrt{1+\alpha}}}=\frac{2\alpha}{\sqrt{1+\alpha}}.$$
 Note that Brannan and Taha~\cite[Theorem 2.1]{branan0} gave the same estimate $|a_2|\leq 2\alpha/\sqrt{1+\alpha}$ for functions $f\in\mathcal{S^*}_{\sigma,\alpha}.$
\end{remark}

\begin{definition} A function $f$ given by (\ref{eq1}) is said to be in the class $K_{\sigma}(\varphi)$,
if $f$ and $F$  satisfy the subordinations
\[1+\frac{zf''(z)}{f'(z)}\prec \varphi(z)\;\;{\text {and}}\;\; 1+\frac{wF''(w)}{F'(w)}\prec \varphi(w).\]
\end{definition}
Note that $K_{\sigma}((1+(1-2\beta)z)/(1-z)))=:K_\sigma(\beta)\;\;
    (0\leq\beta<1).$
\begin{theorem}\label{th3} If  $f\in K_{\sigma}(\varphi)$, then
\[|a_2|\leq \min{\set{\sqrt\frac{ B_1^2+B_1+|B_2-B_1|}{6}, \frac{B_1}{2}}}\] and \[
|a_3|\leq \min{\set{\frac{ B_1^2+B_1+|B_2-B_1|}{6}, \frac{B_1(3B_1+2)}{12}}}.\]
\end{theorem}

\begin{proof} Since $f\in K_{\sigma}(\varphi)$,  there are analytic functions $r, s:\mathbb{D}\rightarrow\mathbb{D}$,  with $r(0)=0=s(0)$,  satisfying
\begin{equation}\label{p1.12}1+\frac{zf''(z)}{f'(z)}=\varphi(r(z))\;\text {and}\; 1+\frac{wF''(w)}{F'(w)} =\varphi(s(z)).\end{equation}
Let $p$ and $q$ be defined as in (\ref{pq}),
then it is clear from (\ref{p1.12}), (\ref{p2}) and (\ref{p3}) that
\begin{equation}\label{p4.12}1+\frac{zf''(z)}{f'(z)}=\varphi\left(\frac{p(z)-1}{p(z)+1}\right)\;\;\text {and}\;\; 1+\frac{wF''(w)}{F'(w)}=\varphi\left(\frac{q(z)-1}{q(z)+1}\right).\end{equation}
It follows from (\ref{p4.12}), (\ref{p5}) and (\ref{p6}) that
\begin{equation}\label{eq11.12}
 2a_2=\frac{1}{2}B_1p_1,
\end{equation}
\begin{equation}\label{eq12.12}
 6a_3 =B_1p_1a_2+\frac{1}{2}B_1\left(p_2-\frac{1}{2}p_1^2\right)+\frac{1}{4}B_2p_1^2,
\end{equation}
\begin{equation}\label{eq15.12}
 -2a_2=\frac{1}{2}B_1q_1
\end{equation}
and
\begin{equation}\label{eq16.12}
 6(2a_2^2-a_3)=-B_1q_1a_2+\frac{1}{2}B_1\left(q_2-\frac{1}{2}q_1^2\right)+\frac{1}{4}B_2q_1^2.
\end{equation}
Now (\ref{eq11.12}) and (\ref{eq15.12}) yield
\begin{equation}\label{eq17.12}
 p_1=-q_1
\end{equation}
and
\begin{equation}\label{eq18.12}
 4a_2=\frac{B_1(p_1-q_1)}{2}.
\end{equation}
From (\ref{eq12.12}), (\ref{eq16.12}), (\ref{eq17.12}) and (\ref{eq11.12}), it follows that
\begin{equation}\label{eq19.12}
 48a_2^2=2B_1^2p_1^2+2B_1(p_2+q_2)+(B_2-B_1)(p_1^2+q_1^2).
\end{equation}
In view of $|p_i|\leq2$ and $|q_i|\leq2$ together with (\ref{eq18.12}) and (\ref{eq19.12}) yield the desired estimate on $a_2$ as asserted in the theorem.
In order to find $a_3$, we subtract (\ref{eq12.12}) from (\ref{eq16.12}) and use (\ref{eq17.12}) to obtain
\begin{equation}\label{eq19.22}
 -12a_3=-12a_2^2+\frac{B_1(q_2-p_2)}{2}.
\end{equation}
 Now a computation using (\ref{eq19.12}) and (\ref{eq19.22}) leads to
\begin{equation}\label{eq19.32}
 -48a_3=2B_1^2p_1^2-4B_2p_2+(B_1-B_2)(p_1^2+q_1^2).
\end{equation}
From (\ref{eq18.12}) and (\ref{eq19.22}), it follows that
\begin{equation}\label{eq19.42}
-12a_3=\frac{B_1(q_2-p_2)}{2}-\frac{3(p_1-q_1)^2B_1^2}{16}.
\end{equation}
 Now (\ref{eq19.32}) and (\ref{eq19.42}) yield the desired estimate on $a_3$ as asserted in the theorem.
\end{proof}
\begin{remark}If $f\in K_\sigma(\beta)\;\;(0\leq \beta<1),$ then theorem~\ref{th3} gives
\[|a_2|\leq \min{\set{\sqrt\frac{ (1-\beta)(3-2\beta)}{3}, 1-\beta}}=1-\beta\] and \[
|a_3|\leq \min{\set{\frac{ (1-\beta)(3-2\beta)}{3}, \frac{(1-\beta)(4-3\beta)}{3}}}=\frac{(1-\beta)(3-2\beta)}{3},\] which improves the Brannan and Taha's~\cite[Theorem 4.1]{branan0} estimates $|a_2|\leq \sqrt{1-\beta} $ and $|a_3|\leq 1-\beta$ for functions $f\in K_\sigma(\beta)$.
\end{remark}
\begin{theorem}\label{th12} Let $f\in\sigma$ be given by (\ref{eq1}). If $f\in\mathcal{K}(\varphi)$ and $F\in\mathcal{R}(\varphi)$, then
\[|a_2|\leq \sqrt \frac{3[B_1+|B_2-B_1|]}{8}\] and \[|a_3|\leq\frac{5[B_1+|B_2-B_1|]}{12}.\]
\end{theorem}
\begin{proof} Since $f\in\mathcal{K}(\varphi)$ and $F\in\mathcal{R}(\varphi)$, there exist two analytic functions $r, s:\mathbb{D}\rightarrow\mathbb{D}$,  with $r(0)=0=s(0)$,  such that
\begin{equation}\label{p11.1}1+\frac{zf''(z)}{f'(z)}=\varphi(r(z))\;\text {and}\; F'(w)=\varphi(s(z)).\end{equation}
Let the functions $p$ and $q$ are defined by (\ref{pq}). It is clear that $p$ and $q$ are analytic in $\mathbb{D}$
and $p(0)=1=q(0)$. Also $p$ and $q$  have positive real part in $\mathbb{D}$, and
hence $|p_i|\leq2$ and  $|q_i|\leq2$.
Proceeding as in the proof of Theorem~\ref{th1} it follow from (\ref{p11.1}), (\ref{p5}) and (\ref{p6}) that
$$ 2a_2=\frac{1}{2}B_1p_1,$$
\begin{equation}\label{eq3}
 6a_3-4a_2^2 =\frac{1}{2}B_1\Big(p_2-\frac{1}{2}p_1^2\Big)+\frac{1}{4}B_2p_1^2,
\end{equation}
$$-2 a_2=\frac{1}{2}B_1q_1 $$
and
\begin{equation}\label{eq5}
 3(2a_2^2-a_3)=\frac{1}{2}B_1\Big(q_2-\frac{1}{2}q_1^2\Big)+\frac{1}{4}B_2q_1^2.
\end{equation}
A computation using (\ref{eq3}) and ({\ref{eq5}}), leads to
\begin{equation}\label{eq7}
  a_2^2=\frac{2(p_2+2q_2)B_1+(p_1^2+2q_1^2)(B_2-B_1)}{32}.
\end{equation}
and
\begin{equation}\label{eq8} a_3=\frac{2(3p_2+2q_2)B_1+(3p_1^2+2q_1^2)(B_2-B_1)}{48}.\end{equation}
Now the desired estimates on $a_2$ and $a_3$, follow  from (\ref{eq7}) and (\ref{eq8}) respectively.\end{proof}
\begin{remark} If $f\in\mathcal{K}(\beta)$ and $F\in\mathcal{R}(\beta)$, then from Theorem~\ref{th12} we see that
\[ |a_2|\leq \sqrt {3(1-\beta)}/2\;\;{\text{and}}\;\; |a_3|\leq 5(1-\beta)/6.\]
 In particular if $f\in\mathcal{K}$ and $F\in\mathcal{R}$, then $|a_2|\leq \sqrt 3/2\approx 0.867\;\;{\text{and}}\;\; |a_3|\leq 5/6\approx0.833$.
\end{remark}

\begin{theorem}\label{th21} Let $f\in\sigma$ be given by (\ref{eq1}). If $f\in\mathcal{S}^*(\varphi)$ and $F\in\mathcal{R}(\varphi),$ then
$$|a_2|\leq \frac{\sqrt {5[B_1+|B_2-B_1|]}}{3},\;\;{\text{and} }\;\; |a_3|\leq\frac{7[B_1+|B_2-B_1|]}{9}.$$
\end{theorem}
\begin{proof} Since
   $f\in\mathcal{S}^*(\varphi)$ and $F\in\mathcal{R}(\varphi),$ there exist two analytic functions $r, s:\mathbb{D}\rightarrow\mathbb{D}$,  with $r(0)=0=s(0)$,  such that
\begin{equation}\label{tp1}\frac{zf'(z)}{f(z)}=\varphi(r(z))\;\text {and}\; F'(w)=\varphi(s(z)).\end{equation}
Let the functions $p$ and $q$ be defined as in (\ref{pq}). Then
 {\small\begin{equation}\label{tp4}\frac{zf'(z)}{f(z)}=\varphi\left(\frac{p(z)-1}{p(z)+1}\right)\;\;\text {and}\;\; F'(w)=\varphi\left(\frac{q(w)-1}{q(w)+1}\right).\end{equation}}
 It follow from (\ref{tp4}), (\ref{p5}) and (\ref{p6}) that
  $$a_2=\frac{1}{2}B_1p_1,$$
\begin{equation}\label{teq3}
 2a_3-a_2^2 =\frac{1}{2}B_1\Big(p_2-\frac{1}{2}p_1^2\Big)+\frac{1}{4}B_2p_1^2,
\end{equation}
$$ -2a_2=\frac{1}{2}B_1q_1, $$
\begin{equation}\label{teq5}
 3(2a_2^2-a_3)=\frac{1}{2}B_1\Big(q_2-\frac{1}{2}q_1^2\Big)+\frac{1}{4}B_2q_1^2.
\end{equation}
A computation using (\ref{teq3}) and ({\ref{teq5}}) leads to
\begin{equation}\label{teq7}
  a_2^2=\frac{2(3p_2+2q_2)B_1+(3p_1^2+2q_1^2)(B_2-B_1)}{36}
\end{equation} and
\begin{equation}\label{teq8} a_3=\frac{2(6p_2+q_2)B_1+(6p_1^2+q_1^2)(B_2-B_1)}{36}.\end{equation}
Now the bounds for $a_2$ and $a_3$ are obtained from (\ref{teq7}) and (\ref{teq8}) respectively using the fact that $|p_i|\leq 2$ and $|q_i|\leq 2$.
\end{proof}
\begin{remark}
 If $f\in \mathcal{S}^*(\beta)$ and $F\in\mathcal{R}(\beta)$, then from Theorem~\ref{th21} it is easy to see that \[ |a_2|\leq \sqrt{10(1-\beta)}/3 \;\;{\text{and}}\;\;|a_3|\leq 14(1-\beta)/9.\] In particular if $f\in \mathcal{S}^*$ and $F\in\mathcal{R}$, then $|a_2|\leq \sqrt 10/3\approx1.054 \;\;{\text{and}}\;\;|a_3|\leq14/9\approx 1.56.$
\end{remark}
\begin{theorem}\label{th5} Let $f\in\sigma$ given by (\ref{eq1}). If $f\in\mathcal{S}^*(\varphi)$ and $F\in\mathcal{K}(\varphi),$ then
\[|a_2|\leq \sqrt \frac{B_1+|B_2-B_1|}{2}\] and
 \[|a_3|\leq\frac{B_1+|B_2-B_1|}{2}.\]
\end{theorem}
\begin{proof}
  Assuming $f\in\mathcal{S}^*(\varphi)$ and $F\in\mathcal{K}(\varphi)$ and proceeding in the similar way as in the proof of Theorem~\ref{th12}, it is easy to see that
  $$ a_2=\frac{1}{2}B_1p_1,$$
\begin{equation}\label{te2}
 3a_3-a_2^2 =\frac{1}{2}B_1\Big(p_2-\frac{1}{2}p_1^2\Big)+\frac{1}{4}B_2p_1^2,
\end{equation}
$$ -2a_2=\frac{1}{2}B_1q_1,$$
\begin{equation}\label{te4}
 8a_2^2-6a_3=\frac{1}{2}B_1\Big(q_2-\frac{1}{2}q_1^2\Big)+\frac{1}{4}B_2q_1^2.
\end{equation}
A computation using (\ref{te2}) and (\ref{te4}) leads to
\begin{equation}\label{te7}
 a_2^2=\frac{2(2p_2+q_2)B_1+(2p_1^2+q_1^2)(B_2-B_1)}{24}
\end{equation} and
\begin{equation}\label{te8}
 a_3=\frac{2(8p_2+q_2)B_1+(8p_1^2+q_1^2)(B_2-B_1)}{72}.
\end{equation}
Now using the result $|p_i|\leq2$ and $|q_i|\leq2$, the estimates on $a_2$ and $a_3$ follow from (\ref{te7}) and (\ref{te8}) respectively.
\end{proof}
\begin{remark} Let $f\in\mathcal{S}^*(\beta)$ and $F\in\mathcal{K}(\beta),\;0\leq\beta<1$. Then from Theorem~\ref{th5}, it is easy to see that \[ |a_2|\leq \sqrt {1-\beta}\;\; {\text{and}}\;\; |a_3|\leq 1-\beta.\] In particular if $f\in\mathcal{S}^*$ and $F\in\mathcal{K},$ then $|a_2|\leq 1\;\; {\text{and}}\;\; |a_3|\leq 1$.
\end{remark}

\noindent{\bf Acknowledgements.} The research is supported by a grant from University of Delhi.

\end{document}